\documentclass[11pt]{amsart}
\usepackage{amsfonts,amssymb,amscd,amsmath,enumerate,verbatim}
\usepackage[latin1]{inputenc}
\usepackage{amscd}
\usepackage{latexsym}
\usepackage{pstcol,pst-plot,pst-3d}
\usepackage{mathptmx}
\usepackage{multicol}

\psset{unit=0.7cm,linewidth=0.8pt,arrowsize=2.5pt 4}

\newpsstyle{fatline}{linewidth=1.5pt}
\newpsstyle{fyp}{fillstyle=solid,fillcolor=verylight}
\definecolor{verylight}{gray}{0.97}
\definecolor{light}{gray}{0.9}
\definecolor{medium}{gray}{0.85}

\def\frk{\mathfrak}               

\def\Phi{{\frk N}}
\def\opn#1#2{\def#1{\operatorname{#2}}} 
\opn\chara{char} \opn\length{\ell} \opn\pd{pd} \opn\rk{rk}
\opn\projdim{proj\,dim} \opn\injdim{inj\,dim} \opn\rank{rank}
\opn\depth{depth} \opn\grade{grade} \opn\height{height}
\opn\embdim{emb\,dim} \opn\codim{codim}

\opn\Tr{Tr} \opn\bigrank{big\,rank}
\opn\superheight{superheight}\opn\lcm{lcm}
\opn\trdeg{tr\,deg}
\opn\reg{reg} \opn\lreg{lreg} \opn\ini{in} \opn\lpd{lpd}
\opn\size{size}\opn{\mult}{mult}
\opn\div{div} \opn\Div{Div} \opn\cl{cl} \opn\Cl{Cl}
\opn\Spec{Spec} \opn\Supp{Supp} \opn\supp{supp} \opn\Sing{Sing}
\opn\Ass{Ass} \opn\Min{Min}
\opn\Ann{Ann} \opn\Rad{Rad} \opn\Soc{Soc}
\opn\Syz{Syz} \opn\Im{Im} \opn\Ker{Ker} \opn\Coker{Coker}
\opn\Am{Am} \opn\Hom{Hom} \opn\Tor{Tor} \opn\Ext{Ext}
\opn\End{End} \opn\Aut{Aut} \opn\id{id} \opn\ini{in}

\opn\nat{nat}
\opn\pff{pf}
\opn\Pf{Pf} \opn\GL{GL} \opn\SL{SL} \opn\mod{mod} \opn\ord{ord}
\opn\Gin{Gin}
\opn\Hilb{Hilb}\opn\adeg{adeg}\opn\std{std}\opn\ip{infpt}
\opn\Pol{Pol}
\opn\sat{sat}
\opn\Var{Var}
\opn\Gen{Gen}

\opn\aff{aff} \opn\con{conv} \opn\relint{relint} \opn\st{st}
\opn\lk{lk} \opn\cn{cn} \opn\core{core} \opn\vol{vol}
\opn\link{link} \opn\star{star}
\opn\gr{gr}

\def\pot#1#2{#1[\kern-0.28ex[#2]\kern-0.28ex]}

\opn\dirlim{\underrightarrow{\lim}}
\opn\inivlim{\underleftarrow{\lim}}
\let\union=\cup
\let\sect=\cap

\let\tensor=\otimes
\let\iso=\cong
\let\Union=\bigcup
\let\Sect=\bigcap

\let\to=\rightarrow

\def\Implies{\ifmmode\Longrightarrow \else
        \unskip${}\Longrightarrow{}$\ignorespaces\fi}
\def\implies{\ifmmode\Rightarrow \else
        \unskip${}\Rightarrow{}$\ignorespaces\fi}
\def\iff{\ifmmode\Longleftrightarrow \else
        \unskip${}\Longleftrightarrow{}$\ignorespaces\fi}

\let\:=\colon
\newtheorem{Theorem}{Theorem}[section]
\newtheorem{Lemma}[Theorem]{Lemma}
\newtheorem{Corollary}[Theorem]{Corollary}
\newtheorem{Proposition}[Theorem]{Proposition}

\newtheorem{Definition}[Theorem]{Definition}

\let\epsilon\varepsilon
\let\phi=\varphi
\let\kappa=\varkappa
\def\qed{\ifhmode\textqed\fi
      \ifmmode\ifinner\quad\qedsymbol\else\dispqed\fi\fi}
\def\textqed{\unskip\nobreak\penalty50
       \hskip2em\hbox{}\nobreak\hfil\qedsymbol
       \parfillskip=0pt \finalhyphendemerits=0}
\def\dispqed{\rlap{\qquad\qedsymbol}}

\opn\dis{dis}
\def\pnt{{\raise0.5mm\hbox{\large\bf.}}}

\opn\Lex{Lex}

\newcommand{\inD}[1][\relax]{\def\argone{#1}\def\temprelax{\relax}
  \ifx\argone\temprelax\right.\else\,\middle|#1\right.{}\fi}

\newif\ifbinary
\binarytrue

\begin{document}

\title{Cohen--Macaulay binomial edge ideals}

\author{Viviana Ene, J\"urgen Herzog and  Takayuki Hibi}

\address{Viviana Ene, Faculty of Mathematics and Computer Science, Ovidius University, Bd.\ Mamaia 124,
 900527 Constanta, Romania} \email{vivian@univ-ovidius.ro}

\address{J\"urgen Herzog, Fachbereich Mathematik, Universit\"at Duisburg-Essen, Campus Essen, 45117
Essen, Germany} \email{juergen.herzog@uni-essen.de}

\address{Takayuki Hibi, Department of Pure and Applied Mathematics, Graduate School of Information Science and Technology,
Osaka University, Toyonaka, Osaka 560-0043, Japan}
\email{hibi@math.sci.osaka-u.ac.jp}

\thanks{}

\begin{abstract}
We study the depth of classes of binomial edge ideals and classify all closed graphs whose binomial edge ideal is Cohen--Macaulay
\end{abstract}
\maketitle

\section*{Introduction}
Binomial edge ideals were introduced in \cite{HHHKR}. They appear independently, and at about the same time, also in the paper \cite{Oh}. In simple terms, a binomial edge ideal is just an ideal generated by an arbitrary collection of $2$-minors of a $2\times n$-matrix whose entries are all
indeterminates. Thus the generators of such an ideal are of the form $f_{ij}=x_iy_j-x_jy_i$ with $i<j$. It is then natural to associate with such an ideal the graph $G$ on the vertex set $[n]$  for which $\{i, j\}$ is an edge if and only if $f_{ij}$ belongs to our ideal. This explains the naming for this type of ideals. The binomial edge ideal of graph $G$ is denoted by $J_G$. In \cite{HHHKR} the relevance of this class of ideals for algebraic statistics is explained.

The goal of this paper is to characterize Cohen--Macaulay binomial edge ideals for simple graphs with vertex set $[n]$. Similar to  ordinary edge ideals which were introduced by Villarreal \cite{V}, a general classification of Cohen--Macaulay binomial edge ideals seems to be hopeless. Thus we have to restrict our attention to special classes of graphs. In Section 1 we first consider the  class of chordal graphs with  the property that any two maximal cliques of it intersect in at most one vertex. These graphs include of course all forests. We show in Theorem~\ref{depthchordal} that  for these graphs we have  $\depth S/J_G=n+c$, where $n$ is the number of vertices of $G$ and $c$ is the number of connected components of $G$. As an application we show that the binomial edge ideal of a forest is Cohen--Macaulay if and only if each of its connected components  is a path graph, and this is the case if and only if $S/J_G$ is a complete intersection.

In Section 3 we  use the results of Section 2 to give in Theorem~\ref{classification} a complete characterization of all closed graphs whose binomial edge ideal is Cohen--Macaulay. Surprisingly this is the case if and only if its  initial ideal is Cohen--Macaulay. Even more is true: if for a closed graph $G$, the ideal $J_G$ is Cohen--Macaulay, then the graded Betti numbers of $J_G$ and its initial ideal coincide. For a closed graph whose binomial edge ideal is Cohen--Macaulay,  the Hilbert function and the multiplicity of $S/J_G$ can be easily computed. Then by using the associativity formula of multiplicities in combination with  the information given in \cite{HHHKR} concerning the minimal prime ideals of binomial edge ideals  we deduce  in Corollary~\ref{combformula} certain numerical identities.

The term ``closed graph" is not standard terminology in graph theory. It was introduced in \cite{HHHKR}  to characterize those graphs, which,  for certain labeling of their edges, do have a quadratic Gr\"obner basis  with respect to the lexicographic order induced by $x_1>\cdots >x_n>y_1>\cdots >y_n$. It is easy to see, as shown in \cite{HHHKR}, that any closed graph must be chordal. But
by far not all chordal graphs are closed. In Theorem~\ref{characterization} we give a description of the closed  graphs which is then used in the proof of  Theorem~\ref{classification}.

\section{Classes of chordal graphs with Cohen--Macaulay binomial edge ideal}

Recall that, by a result of Dirac \cite{D} (see also \cite{HH10}), a  graph $G$ is chordal if and only if it admits a {\em perfect elimination order}, that is, its vertices can be labeled $1,\ldots, n$ such that for all $j\in[n]$, the set $C_j=\{i\:\; i\leq j\}$ is a clique of $G$. A clique is simply a complete subgraph of $G$.

There is an equivalent characterization of chordal graphs in terms of its maximal cliques. To describe it we introduce some terminology.  Let $\Delta$ be a simplicial complex. A facet $F$ of $\Delta$ is called a {\em leaf}, if either $F$ is the only facet, or else there exists a facet $G$, called a {\em branch} of $F$, which intersects $F$ maximally. In other words, for each facet $H$ of $\Delta$ with $H\neq F$  one has $H\sect F\subset G\sect F$. Each leaf $F$ has at least one {\em free vertex}, that is, a vertex which belongs only to $F$. On the other hand, if a facet admits a free vertex it needs not to be a leaf.

The simplicial complex $\Delta$ is a called a {\em quasi-forest} if its facets can be ordered $F_1,\ldots,F_r$ such that for all $i>1$ the facet $F_i$ is a leaf of the simplicial complex with facets $F_1,\ldots, F_{i-1}$. Such an order of the facets is called a {\em leaf order}. A connected quasi-forest is called a {\em quasi-tree}.

Now let $G$ be a graph. The collection of cliques of $G$ forms a simplicial complex, called the {\em clique complex} of $G$. It is denoted  $\Delta(G)$. The equivalent statement to Dirac's theorem now says that $G$ is chordal if and only if $\Delta(G)$ is a quasi-forest.

In this section we will compute the depth of $S/J_G$ for  a very special class of chordal graphs.  This class includes all forests. As a consequence it will be shown that a forest has a Cohen--Macaulay binomial edge ideal if and only if all its components are path graphs.

We shall need a few results from \cite{HHHKR}. There in Corollary 3.9 and Corollary 3.3 the following fact is shown:  Suppose that $G$ is connected. Let $S\subset [n]$, and let $G_1,\ldots, G_{c_G(S)}$ be the connected components of $G_{[n]\setminus S}$. For each $G_i$ we denote by $\tilde{G}_i$ the complete graph on the vertex set $V(G_i)$. If there is no confusion possible we simply write $c(S)$ for $c_G(S)$, and set
\[
P_S(G)=(\Union_{i\in S}\{x_i,y_i\}, J_{\tilde{G_1}},\ldots, J_{\tilde{G}_{c(S)}}).
\]
Then $J_G=\Sect_{S\subset [n]}P_S(G)$, and $P_S(G)$ is a minimal prime ideal of $J_G$ if and only if $S=\emptyset$, or $S\neq\emptyset$ and for each $i\in S$ one has $c(S\setminus \{i\})<c(S)$. Moreover, $\height P_S(G)=n+|S|-c(S)$ and hence  $\dim S/J_G=\max\{(n-|S|)+c(S)\;\:\; S\subset [n]\}.$

\begin{Theorem}
\label{depthchordal}
Let $G$ be a  chordal graph on $[n]$ with the property that any two distinct maximal cliques intersect in  at most one  vertex. Then
$\depth S/J_G=n+c$, where $c$ is the number of connected components of $G$.

Moreover,  the following conditions are equivalent:
\begin{enumerate}
\item[{\em (a)}] $J_G$ is unmixed.
\item[{\em (b)}] $J_G$ is Cohen--Macaulay.
\item[{\em (c)}] Each vertex of $G$ is the intersection of at most two maximal cliques.
\end{enumerate}
\end{Theorem}

\begin{proof}
Let $G_1,\ldots,G_c$ be the connected components of $G$ and set $S_i=K[\{x_j,y_j\}_{j\in G_i}]$. Then $S/J_G\iso S_1/J_{G_1}\tensor \cdots \tensor S_c/J_{G_c}$, so that $\depth S/J_G=\depth S_1/J_{G_1}+\cdots + \depth S_c/J_{G_c}$. Thus in order to prove the desired result, we may assume that $G$ is connected.

Let  $\Delta(G)$ be the clique complex of $G$ and let $F_1,\ldots,F_r$ be a leaf order on the facets of $\Delta(G).$ We make induction
on $r.$ If $r=1,$ then $G$ is a simplex and the statement is true. Let $r>1$; since $F_r$ is a leaf, there exists a unique vertex, say $i\in F_r$, such that $F_r\sect F_j=\{i\}$  for some $j$. Let $F_{t_1},\ldots,F_{t_q}$ be the
facets of $\Delta(G)$ which intersect the leaf $F_r$ in the vertex $i.$

Let ${\mathcal M}(G)$ denote the set of all sets $S\subset [n]$ such that $P_S(G)$ is a minimal prime ideal of $J_G$. We have  $J_G=Q_1\sect Q_2$ where  $Q_1=\Sect_{S\in \mathcal{M}(G),\; i\not\in S} P_S(G)$ and  $Q_2=\Sect_{S\in \mathcal{M}(G),\; i\in S} P_S(G)$.

Consider  the exact sequence
\begin{equation}\label{eqchordal}
0 \to S/J_G\to S/Q_1\oplus S/Q_2\to S/(Q_1+Q_2)\to 0.
\end{equation}
The ideal $Q_1$ is the binomial edge ideal associated with the graph $G^{\prime}$ which is obtained from $G$ by replacing the facets
$F_{t_1},\ldots,F_{t_q},$ and $F_r$ by the clique on the vertex set $F_r\union(\Union\limits_{j=1}^q F_{t_j}).$ Note that
$G^{\prime}$ is a connected chordal graph which has again the property that any two cliques intersect in at most  one vertex, and it has a smaller number of cliques than $G.$ Therefore, by induction, we have $\depth(S/Q_1)=\depth(S/J_{G^{\prime}})=n+1.$

In order to determine $Q_2$ we first observe that for all $S\subset [n]$ with $i\in S$ we have that $P_S(G)=(x_i,y_i)+P_{S\setminus\{i\}}(G'')$, where $G''$ is the restriction of $G$ to the vertex set $[n]\setminus\{i\}$. From  this we conclude that $Q_2=(x_i,y_i)+J_{G''}$. Let $S_i$ be the polynomial ring $S/(x_i,y_i)$. Then $S/Q_2\iso S_i/J_{G''}$. Hence, since $G''$ is a graph on $n-1$ vertices and with $q+1$ components satisfying the conditions of the theorem, our induction hypothesis implies that $\depth S/Q_2 =(n-1)+q+1=n+q$.

Next we observe that $Q_1+Q_2=J_{G'}+((x_i,y_i)+J_{G''})=(x_i,y_i)+J_{G'}$. Thus $S/(Q_1+Q_2)\iso S_i/J_H$ where $H$ is obtained form $G'$ by replacing the clique on the vertex set $F_r\union(\Union\limits_{j=1}^q F_{t_j})$ by the clique on the vertex set $F_r\union(\Union\limits_{j=1}^q F_{t_j})\setminus \{i\}$. Thus our induction hypothesis implies that $\depth S/(Q_1+Q_2)=n$. Hence the depth lemma applied to the exact sequence (\ref{eqchordal}) yields the desired conclusion concerning the depth of $S/J_G$.

For the proof of the equivalence of statements (a), (b), and (c), we may again assume that $G$ is connected. Let $J_G$ be unmixed. Then
$\dim(S/J_G)=n+1$ since $J_G$ has a minimal prime of dimension $n+1,$ namely $P_\emptyset(S).$ Since $\depth(S/J_G)=n+1$, it follows
that $J_G$ is Cohen-Macaulay, whence (a) \implies (b). The converse, (b) \implies (a), is well known.

(a) \implies (c): Let us assume that there is a vertex $i$ of $G$ where at least three cliques intersect. Then, for $S=\{i\}$,
 we get a  minimal prime $P_S(G)$ of $J_G$ of height strictly smaller than $n-1,$ which is in contradiction with the hypothesis on
 $J_G.$

(c) \implies (a): Let $\{i_1,\ldots,i_{r-1}\}$ be the intersection vertices of the maximal cliques of $G$, and  $P_S(G)$  a minimal
prime of $J_G$. Let $H_1,\ldots,H_t$ be the connected components of $G_{[n]\setminus S}.$ Suppose that there exists $i\in S\setminus
\{i_1,\ldots,i_{r-1}\}$. We have $c(S\setminus\{i\})<c(S)$. This implies that there exists $H_a, H_b$, two connected components of
$G_{[n]\setminus S},$ such that $i$ is connected to $H_a$ and $H_b.$ Let $u\in V(H_a)$ and $v\in V(H_b)$ such that $\{i,u\}$ and
$\{i,v\}$ are edges of $G.$ Since $i\in S\setminus \{i_1,\ldots,i_{r-1}\},$ it follows that $u,v$ and $i$ belong to the same clique
of $G$, which implies that $\{u,v\}$ is an edge of $G.$ Therefore, $H_a$ and $H_b$ are connected, a contradiction. By induction on the cardinality of $S$  we see that $c(S)=|S|+1$. Therefore, all the minimal primes of $J_G$ have the same height.
\end{proof}

As a consequence of Theorem~\ref{depthchordal} we obtain the following

\begin{Corollary}
\label{depthforest}
Let $G$ be a forest on the vertex set $[n].$ Then $\depth(S/J_G)=n+c,$ where $c$ is the number of the connected components of $G.$
Moreover, the following conditions are equivalent:
\begin{itemize}
	\item [(a)] $J_G$ is unmixed;
	\item [(b)] $J_G$ is Cohen-Macaulay;
	\item [(c)] $J_G$ is a complete intersection;
	\item [(d)] Each component of $G$ is a path graph.
\end{itemize}
\end{Corollary}

\begin{proof} The implications (c)\implies (b)\implies (a) are obvious, while  (a )\implies (d) follows from Theorem~\ref{depthchordal}. For the proof of (d)\implies (c) we may assume that $G$ is a path, and the vertices are labeled in such a way such that $E(G)=\{ \{i,i+1\}\:\, i=1,\ldots,n-1\}$. Then $\ini_<(J_G)=(x_1y_2, x_2y_3,\ldots,x_{n-1}y_n)$, where $<$ is the lexicographic order induced by $x_1>x_2>\cdots >x_n>y_1>y_2>\cdots >y_n$. Since the initial ideal of $J_G$ is a complete intersection, $J_G$ itself is a complete intersection.
\end{proof}

The depth formula that we proved in Theorem~\ref{depthchordal} is not valid for arbitrary chordal graphs. For example for the graph $G$ displayed in Figure 1 we have $\depth S/J_G=5$ (and not $6$ as one would expect  by Theorem~\ref{depthchordal}). It is also an example of a graph for which $J_G$ is unmixed but not Cohen--Macaulay.

\begin{figure}[hbt]
\begin{center}
\psset{unit=1cm}
\begin{pspicture}(2.5,0.5)(7,2.5)
\pspolygon(3.4,2)(4.7,2)(4.7,1)
 \pspolygon(4.7,1)(4.7,2)(6,2)
\psline(4.7,1)(6,0.3)
 \psline(4.7,2)(6,0.3)
 \rput(3.4,2){$\bullet$} \rput(4.7,2){$\bullet$}
 \rput(4.7,1){$\bullet$} \rput(6,2){$\bullet$}
 \rput(6,0.3){$\bullet$}
\end{pspicture}
\end{center}
\caption{}\label{Fig1}
\end{figure}

\section{Closed graphs}

In  \cite{HHHKR} the concept of closed graphs was introduced. In that paper  a simple graph $G$  on the vertex set $[n]$ is called {\em closed with respect to the given labeling},  if the following condition is satisfied:

\begin{itemize}
\item For all   $\{i,j\}, \{k,l\}\in E(G)$ with $i<j$ and $k<l$ one has  $\{j,l\}\in E(G)$ if $i=k$ but $j\neq \ell$, and $\{i,k\}\in E(G)$ if $j=l$ but $i\neq k$.
\end{itemize}

The definition was motivated by the following result \cite[Theorem 1.1]{HHHKR}: $G$ is closed with respect to the given labeling,  if and only if $J_G$ has a quadratic Gr\"obner basis with respect to the  lexicographic order induced by $x_1>x_2>\cdots >x_n> y_1>\cdots >y_n$.

It is shown in \cite[Proposition 1.4]{HHHKR} that the graph $G$ on $[n]$ is closed  with respect to the given labeling,  if and only if for any two integers $1\leq i<j\leq n$ the shortest walk $\{i_1,i_2\}, \{i_2,i_3\},\ldots, \{i_{k-1}, i_{k}\}$ between $i$ and $j$ has the property that $i=i_1<i_2<\cdots < i_k=j$. In particular, for each $i<n$ one has that $\{i,i+1\}\in E(G)$.

\begin{Definition} We say a graph is {\em closed} if there exists a labeling for which it is closed.
\end{Definition}


It arises the question to characterize the closed graphs. It is  known from \cite[Proposition 1.2]{HHHKR} that if  $G$ is closed, then $G$ is chordal.

\begin{Theorem}
\label{characterization} Let $G$ be a graphs on $[n]$. The following conditions are equivalent:
\begin{enumerate}
\item[{\em (a)}] $G$ is closed;
\item[{\em (b)}] there exists a labeling of $G$ such that all facets of $\Delta(G)$ are intervals $[a,b]\subset [n]$.
\end{enumerate}
Moreover, if the equivalent conditions hold and the facets $F_1,\ldots,F_r$ of $\Delta(G)$ are labeled such that $\min(F_1)<\min(F_2)<\cdots <\min(F_r)$, then $F_1,\ldots,F_r$ is a leaf order of $\Delta(G)$.
\end{Theorem}

\begin{proof} (a)\implies (b): Let $G$ be a closed graph on $[n]$ and  $F=\{j\:\; \{j,n\}\in E(G)\}$, and let $k=\min\{j\:\; j\in F\}$. Then $F=[k,n]$. Indeed, if $j\in F$ with $j<n$, then as observed above, it follows that $\{j,j+1\}\in E(G)$, and then because $G$ is closed we see that since  $\{j,n\}\in  E(G)$, then also $\{j+1,n\}\in  E(G)$.  Thus $j+1\in F$.

Next observe that $F$ is a maximal clique of $G$, that is, a facet of $\Delta(G)$. First of all it is a clique, because $i,j\in F$ with $i<j<n$, then, since $\{i,n\}$ and $\{j,n\}$ are edges of $G$, it follows that $\{i,j\}$ is an edge as well, since $G$ is closed. Secondly, it is maximal, since $\{j,n\}\not\in E(G)$, if $j\not\in F$.

Let $H\neq F$ be a facet of $\Delta(G)$ with $H\sect F\neq \emptyset$, and let $\ell=\max\{j\:\; j\in H\sect F\}$. We claim that $H\sect F=[k, \ell]$. There is nothing to prove if $k=\ell$.  So now suppose that $k<\ell$ and let $k\leq t<\ell$  and $s\in H\setminus F$. Then $s,t<\ell$ and $\{s,\ell\}$ and  $\{t,\ell\}$ are edges of $G$. Hence since $G$ is closed it follows that $\{s,t\}\in E(G)$. This implies that $s\in H$, as desired.

It follows from the claim that the facet $H$ for which $\max\{j\:\; j\in H\sect F\}$ is maximal, is a branch of $F$. In particular,  $F$ is a leaf.
Let $H\sect F=[k, \ell]$, where $H$ is a branch of $F$, and denote by $G_\ell$  the restriction of $G$ to $[\ell]$. Since $G_\ell$ is again closed and since $\ell<n$, we may assume, by applying induction on the cardinality of the vertex set of $G$, that all facets of $\Delta(G_\ell)$ are intervals. Now let $F'$ be any facet of $\Delta(G)$. If $F=F'$, then $F$ is an interval, and if $F\neq F'$, then, as we have seen above, it follows that $F'\in \Delta(G')$. This yields the desired conclusion.

(b)\implies (a): Let $\{i,j\}$ and $\{k,\ell\}$ be edges of $G$ with $i<j$ and $k<\ell$. If $i=k$, then  $\{i,k\}$ and $\{i,\ell\}$ belong to the same maximal clique, that is, facet of $\Delta(G)$ which by assumption is an interval. Thus if $j\neq \ell$, then $\{j,\ell\}\in E(G)$. Similarly one shows that if $j=\ell$, but $i\neq k$, then $\{i,k\}\in E(G)$. Thus $G$ is closed.

Finally it is obvious that the facets of $\Delta(G)$ ordered according to their minimal elements is a leaf order, because for this order $F_{i-1}$ has maximal intersection with $F_i$ for all $i$.
\end{proof}

\section{Closed graphs with Cohen--Macaulay binomial edge ideal}

With the description of closed graphs given in Theorem~\ref{characterization} it is not hard to classify all closed graphs with Cohen--Macaulay binomial edge ideal.

\begin{Theorem}
\label{classification}
Let $G$ be a connected graph on $[n]$ which is closed with respect to the given labeling. Then the following conditions are equivalent:
\begin{itemize}
\item [(a)] $J_G$ is unmixed;
\item [(b)] $J_G$ is Cohen-Macaulay;
\item [(c)] $\ini_< (J_G)$ is Cohen-Macaulay;
\item [(d)] $G$ satisfies the condition that whenever $\{i,j+1\}$ with $i<j$ and $\{j,k+1\}$ with $j<k$ are edges of $G$, then
	$\{i,k+1\}$ is an edge of $G;$
\item [(e)] there exist integers $1=a_1<a_2<\cdots <a_r<a_{r+1}=n$ and a leaf order of the facets $F_1,\ldots,F_r$ of $\Delta(G)$ such that $F_i=[a_i,a_{i+1}]$ for all $i=1,\ldots,r$.
\end{itemize}
\end{Theorem}

\begin{proof}
We begin by proving (a) $\implies$ (e). By Theorem~\ref{characterization}, $\Delta(G)$ has facets  $F_1,\ldots,F_r$ where each facet is an interval. We may order the intervals $F_i=[a_i,b_i]$  such that $1=a_1<a_2<\cdots <a_r\leq b_r=n$. Since $G$ is connected it follows that $a_{i+1}\leq b_i$ for all $i$. Let $S=[a_r,b_{r-1}]$; then $c(S)=2$, and so height $P_S(G)=n+(b_{r-1}-a_r+1)-2=n+(b_{r-1}-a_r)-1$. On the other hand, $\height P_\emptyset(G)=n-1$, since $G$ is connected. Thus our assumption implies that  $n+(b_{r-1}-a_r)-1= n-1$ which implies that $b_{r-1}=a_r$. Let $G'$ be the graph whose clique complex  $\Delta(G')$ has the facets $F_1,\ldots,F_{r-1}$. Let $P_S(G')$ be a minimal prime ideal of $G'$. Then $b_{r-1}\not\in S$. Therefore, $c_{G'}(S)=c_G(S)$, and hence $P_S(G)$ is a minimal prime ideal of $J_G$ of same height as $P_S(G')$. Thus we conclude that $J_{G'}$ is unmixed as well. Induction on $r$ concludes the proof.

In the sequence of implications (e) $\implies$ (d) $\implies$ (c) $\implies$ (b) $\implies$ (a), the second follows from the proof of \cite[Proposition
1.6.]{HHHKR}, and the third and the fourth are well known for any ideal.

We prove (e) $\implies$ (d). Let $i<j<k$ be three vertices of $G$ such that $\{i,j+1\}$ and $\{j,k+1\}$ are edges of $G.$ Then $i$ and $j+1$
belong to the same facet of $\Delta(G)$, let us say to $F_{\ell}$. Then $k+1$ must belong to $F_{\ell}$ as well  since it is adjacent to $j.$ Therefore, the condition from (d) follows.
\end{proof}

Closed graphs with Cohen-Macaulay binomial edge ideal  have the following nice property.

\begin{Proposition}
\label{betticm}
Let $G$ be a closed graph with Cohen--Macaulay binomial edge ideal. Then $\beta_{ij}(J_G)=\beta_{ij}(\ini(J_G))$ for all $i$ and $j$.
\end{Proposition}

\begin{proof}
For a graded $S$-module $W$ we denote by $B_W(s,t)=\sum_{i,j}\beta_{ij}(W)s^it^j$ the Betti polynomial of $W$.

Since $\ini(J_G)$ is Cohen--Macaulay, it follows from Theorem~\ref{classification} that $[n]=\Union_{k=1}^r[a_k,a_{k+1}]$ with $1=a_1 <a_2<\cdots <a_r<a_{r+1}=n$ and  such that each $F_k\:=G_{[a_k, a_{k+1}]}$ is a clique. It follows that $\ini(J_G)$ is minimally generated by the set of monomials $\Union_{k=1}^rM_k$ where $M_k=\{x_iy_j\:\; a_k\leq i<j\leq a_{k+1}\}$ for all $k$. Since for all $i\neq j$ the  monomials of $M_i$ and $M_j$ are monomials in disjoint sets of variables, it follows that $\Tor_k(S/(M_i), S/(M_j))=0$ for all $i\neq j$ and all $k>0$. Form this we conclude that
\[
B_{S/\ini(J_G)}(s,t)=\prod_{i=1}^rB_{S/(M_i)}(s,t).
\]
Since $\Tor_k(S/(M_i), S/(M_j))=0$ for all $k>0$, and  since $\ini_<(J_{F_i})=(M_i)$  for all $i$, we see  that $\Tor_k(S/J_{F_i}, S/J_{F_j})=0$ for all $k>0$ as well. Thus we have
\[
B_{S/J_G}(s,t)=\prod_{i=1}^rB_{S/J_{F_i}}(s,t).
\]
Hence it remains to be shown that if $G$ is a clique, then $\beta_{ij}(J_G)=\beta_{ij}(\ini(J_G))$ for all $i$ and $j$.
By the subsequent Lemma~\ref{linear} and by  Fr\"oberg's theorem \cite{F} we have that $\ini_<(J_G)$ has a $2$-linear resolution. Therefore $J_G$ has a $2$-linear resolution as well. Thus for $J_G$ and for $\ini_<(J_G)$, the Hilbert function of the ideal determines the Betti numbers. It is well-known that $S/\ini_<(J_G)$ and $S/J_G$ have the same Hilbert function. Hence we conclude that the  (graded) Betti numbers of $J_G$ and $\ini_<(J_G)$ coincide.
\end{proof}

\begin{Lemma}
\label{linear}
Let $G$ be a finite bipartite graph on
$\{x_1, \ldots, x_n\} \cup \{y_n, \ldots, y_n\}$
with the edges $\{x_i, y_j\}$ with $1 \leq i \leq j \leq n$.
Then the complementary graph $\bar{G}$
of $G$ is a chordal graph.
\end{Lemma}

\begin{proof}
Let $X = \{x_1, \ldots, x_n\}$
and $Y = \{y_1, \ldots, y_n\}$.
Let $C$ be a cycle of $\bar{G}$ of length at least $5$.
Then $C$ contains either three vertices belonging to $X$
or three vertices belonging to $Y$.
Since $\{x_i, x_j\}$ and $\{y_i, y_j\}$ are edges of
$\bar{G}$ for all $i \neq j$,
it follows that $C$ possesses a chord.

Now, let $C = (a, b, c, d)$ be a cycle of $\bar{G}$
of length $4$.  If $a, c \in X$ or $b, d \in X$
or $a, c \in Y$ or $b, d \in Y$, then $c$ possesses a chord.
Suppose that
$a \in X$, $c \in Y$, $b \in X$ and $d \in Y$,
say, $C = (x_i, x_j, y_k, y_\ell)$.
Then $k < j$ and $\ell < i$.
If $j < i$, then $k < i$.  Thus $\{x_i, y_k\}$
is a chord of $C$.
If $i < j$, then $\ell < j$.  Thus $\{x_j, y_\ell\}$
is a chord of $C$.
Hence $\bar{G}$ is chordal, as desired.
\end{proof}

Proposition~\ref{betticm} yields

\begin{Corollary}
\label{cmtype}
Let $G$ be a closed graph with Cohen--Macaulay binomial edge ideal, and assume that $F_1,\ldots, F_r$ are the facets of $\Delta(G)$ with $k_i=|F_i|$ for $i=1,\ldots,r$. Then the Cohen--Macaulay type of $S/J_G$ is equal to $\prod_{i=1}^r(k_i-1)$.
In particular, $S/J_G$ is Gorenstein if and only if $G$ is a path graph.
\end{Corollary}

\begin{proof} Due to Proposition~\ref{betticm} it suffices to show that if $G$ is a clique on $[n]$ (with $n\geq 2$), then the Cohen--Macaulay type of $S/J_G$ is equal to $n-1$. In this particular case, $J_G$ is the ideal of $2$-minors of a $2\times n$-matrix whose resolution is given by the Eagon--Northcott complex. The type of $S/J_G$ is the last Betti number in the resolution, which is $n-1$.
\end{proof}

Let $G$ be a closed graph with Cohen--Macaulay binomial edge ideal, and assume that $F_1=[a_1,a_2],\ldots, F_r=[a_r,a_{r+1}]$, where $1=a_1< a_2< 
\cdots < a_r < a_{r+1}=n,$
 are the facets of $\Delta(G)$ and $k_i=|F_i|$ for $i=1,\ldots,r$. By using the well-known fact that $S/J_G$ and $S/\ini(J_G)$ have the same Hilbert series, one easily gets the Hilbert series of $S/J_G,$
 \[
H_{S/J_G}(t)= \frac{\prod_{i=1}^r[(k_i-1)t+1]}{(1-t)^{n+1}}.
\]
In particular, the multiplicity of $S/J_G$ is $e(S/J_G)=k_1\cdots k_r$ and the $a$-invariant is $a(S/J_G)=r-n-1.$

By using the associativity formula for multiplicities we obtain a different expression for the multiplicity as the one given above. This will be a consequence of

\begin{Proposition}
\label{primemult}
$P_S(G)$ is a minimal prime of $J_G$ if and only if $S$ is empty or of the form $S=\{a_{j_1},\ldots,a_{j_s}\}$ for some
$2\leq j_1< j_2<\cdots <j_s\leq r$ such that $a_{j_{q+1}}-a_{j_q}\geq 2$ for all $1\leq q\leq s-1.$

In this case, the multiplicity of $S/P_S(G)$ is
\[
e(S/P_S(G)) =(a_{j_1}-1)(a_{j_2}-a_{j_1}-1)\cdots (a_{j_s}-a_{j_{s-1}}-1)(n-a_{j_s}).
\]
\end{Proposition}

\begin{proof} For any $s,$ if $S=\{a_{j_1},\ldots,a_{j_s}\}$ with $1\leq j_1< j_2<\cdots
<j_s\leq r-1$ such that $a_{j_{q+1}}-a_{j_q}\geq 2$ for all $1\leq q\leq s-1,$ the number of the connected components of the
restriction $G_{[n]\setminus S}$ of $G$ is $s+1.$ This implies that for such $S$, $P_S(G)$ is a minimal prime ideal of $J_G$.

Conversely,  let $S\neq \emptyset$, $S\subset [n],$ such that $P_S(G)$ is a minimal prime of $G.$ In the first place we claim that
$S$ is contained in $\{a_2,\ldots,a_r\}.$ Indeed, let us suppose that there exists $j\in S\setminus \{a_2,\ldots,a_r\}$, and let
$H_1,\ldots,H_t$ be the connected components of $G_{[n]\setminus S}$. Since $P_S(G)$ is a minimal prime, we have
$c(S\setminus\{j\})< c(S)$. This implies that there exists some integers $a\neq b$ such that $j$ is connected to $H_a$ and $H_b.$
Let $u\in V(H_a)$ and $v\in V(H_b)$ such that $\{u,j\}$ and $\{v,j\}$ are edges of $G.$ Then $u,v$, and $j$ belong to the same clique
of $G$, thus $\{u,v\}$ is an edge of $G$ and $H_a, H_b$ are connected, which is impossible. Consequently, $S$ is a subset of
$\{a_2,\ldots,a_r\}.$ Let $S=\{a_{j_1},\ldots,a_{j_s}\}$ with $2\leq j_1< j_2<\cdots <j_s\leq r$ and assume that there exists
$1\leq q\leq s-1$ such that $a_{j_{q+1}}=a_{j_q}+1.$ This means that $F_{j_q}=\{a_{j_q}, a_{j_q}+1\}.$ In this case it is easy to
check that $c(S\setminus\{a_{j_q}\})=c(S)$, which leads to a contradiction with  the minimality of $P_S(G)$.

The formula for the multiplicity follows easily if we recall that the multiplicity of $J_C$ is $m$ if $C$ is a clique with $m$ vertices.
\end{proof}

By comparing the two formulas for the multiplicity of $S/J_G$, we  get  the following

\begin{Corollary}
\label{combformula}
Let $b_1,\ldots,b_r\geq 1$ be some integers.
Then
\[
(b_1+1)\cdots (b_r+1)=1+\sum_{i=1}^r b_i+
\]
\[
+ \sum_{s=1}^{r-1}\sum_{1\leq j_1<\cdots < j_s\leq r-1}\left[(b_1+\cdots + b_{j_1})\prod_{i=1}^{s-1}(b_{j_i+1}+\cdots +b_{j_{i+1}}-1)(b_{j_s+1}+\cdots + b_r)\right].
\]
In particular, we have the following identity
\[
2^r=\sum_{s=0}^{\left\lfloor \frac{r}{2}\right\rfloor}\sum_{(x_1,\ldots,x_{s+1}) \in P(r-s+1|s+1)}x_1\cdots x_{s+1},
\]
where $P(r-s+1|s+1)$ stands for the set of all partitions of $r-s+1$ with $s+1$ parts.

\end{Corollary}

\bigskip
In Proposition~\ref{betticm} we have seen that  for a closed graph $G$,  whose  binomial edge ideal $J_G$ is Cohen--Macaulay,  the graded Betti numbers of $J_G$ and $\ini_<(J_G)$ coincide. Computational evidence indicates that the graded Betti numbers of $J_G$ and $\ini_<(J_G)$ coincide for all closed graphs. More generally, we conjecture that if $G$ is a chordal graph whose clique complex $\Delta(G)$ has a leaf order $F_1,\ldots,F_r$ such that   $F_{i-1}$ is the unique branch of $F_i$ for $i=2,\ldots, r$, then $J_G$ and $\ini_<(J_G)$ have the same graded Betti numbers.

We call  chordal graphs with the above property on the leaf order {\em  chain of cliques}. Each closed graphs is a chain of cliques  as we have seen in Theorem~\ref{characterization}.  The converse is not true, as the following example shows:

\begin{figure}[hbt]
\begin{center}
\psset{unit=1.5cm}
\begin{pspicture}(2.5,0.5)(7,1.5)
\pspolygon(4.7,0.3)(3.7,0.3)(4.0,1.0)(4.7,1.3)(5.4,1.0)(5.7,0.3)(4.7,0.3)
\psline(4.7,0.3)(4.0,1.0)
\psline(4.7,0.3)(4.7,1.3)
\psline(4.7,0.3)(5.4,1.0)
\rput(4.7,0.3){$\bullet$}
\rput(4.0,1.0){$\bullet$}
\rput(4.7,1.3){$\bullet$}
\rput(5.4,1.0){$\bullet$}
\rput(3.7,0.3){$\bullet$}
\rput(5.7,0.3){$\bullet$}
\end{pspicture}
\end{center}
\caption{}\label{Fig1}
\end{figure}

Based on explicit calculations and general arguments in special cases  we believe that in general for all graphs $G$ the extremal Betti numbers (see \cite{BCP}) of $J_G$ and $\ini_<(J_G)$ coincide.

{}

\end{document}